\documentclass{birkjour}
 \newtheorem{thm}{Theorem}[section]

 \theoremstyle{definition}
 
 \theoremstyle{remark}

 \numberwithin{equation}{section}

\usepackage{url}
\usepackage{amsmath,amssymb}
\usepackage{amsfonts}
\usepackage{subfig} 
\usepackage{caption}
\begin{document}
%--------------------------------------%
\title[The Wade Formula and the oil price fluctuation]
 {The Wade Formula and the oil price fluctuation: An Optimal Control  Theory  approach}
 
%----------Authors------%
\author[Mendy et al.]{Pierre Mendy, Babacar M. Ndiaye, Diaraf Seck, Idrissa Ly}
\address{Laboratory of Mathematics of Decision and Numerical Analysis.\br
Cheikh Anta Diop University.  
BP 45087, 10700. Dakar, Senegal.}

\email{pierre.mendy@ucad.edu.sn, babacarm.ndiaye@ucad.edu.sn, \\ diaraf.seck@ucad.edu.sn, idrissa.ly@ucad.edu.sn}

\thanks{This work was completed with the support of the NLAGA project}

\keywords{Wade Formula, optimal control, options, volatility, oil prices}

\date{November 21, 2022}

\begin{abstract} 
By considering the Wade Formula, we propose a model to study the evolution of the oil price per barrel. Our model shows that the policy of diversification of the energy is to be supported. This model is proposed to see how it is possible to control parameters so that the oil price should decrease.
\end{abstract}

%----%
\maketitle

\vspace{-1.cm}
%---%
\section{Introduction}\label{intro}
The rapid increase in oil prices between 2002 and 2008 and their sharp decline in the second half of 2008 and 2014, (see \cite{baum2016}, \cite{ArBlan2015}),  and the consequence of covid19 and Ukraine conflic on oil price,   (\cite{ Wor2020}, \cite{Tagh2020}, \cite{bbc}),   has renewed the interest in the causes of oil price fluctuation [figure 1, curve(d)], and the effects of energy prices on the macroeconomy. \\
A large of studies prove that the oil fluctuations have a considerable consequence in economic activity (\cite{ThYo2015},\cite{JrSa2004},\cite{OlEk2013}, \cite{FuSo2015},\cite{FtGrTe2015}, \cite{DeReMa2009}, \cite{Taghetal2019d},  \cite{Taghetal2013b}, \cite{Taghetal2016a}).\\
While the oil price shock is asymmetric between oil exporters countries and oil importers countries (\cite{SgAp2016}, \cite{Natetal1998}, \cite{Taghetal2016c}), 
this asymmetric oil shock had inspired the Wade Formula. \\
Between 2005 and 2006, the full professor in Economics Abdoulaye Wade, former President of the Republic of Senegal, proposed a formula related to the evolution of the oil price per barrel. This formula translates the super-profits generated by the oil companies (selling in the world) and countries which produce oil.\\
At first, from the Wade Formula, we describ in section {\ref{model}}, the model. We resolve and show a necessary optimal condition which could give some hints for the orientations of the energy policies in section {\ref{modelres}}. We propose some simulations and interpretations of the obtained results, section {\ref{numericalres}}.  As conclusion in section {\ref{ccl}}, we present some variations of the obtained results and politics implications.

%---%
\section{Model description}\label{model}
The formula should be stated as follows:
Let be the time $t\in [T_0, T]$,  where $T_0$  is the initial state. 
Since it is possible to find the structure of the oil price/barrel without financial speculation, President Wade studied and proposed a reasonable price (win-win price) $P_0 = 29\$ $. For more details, see \cite{Wade}. \\
The Wade Formula for super profit is defined by the equation (\ref{Sp1})
\begin{equation}\label{Sp1}
(p_t - 29)q_t = S_t
\end{equation}
where $p_t$ represents the price which is applied in the world for buying a barrel of oil, $q_t $ is the quantity of oil bought by the countries, and $S_t$ is the super profit for the oil companies or the super coast supported by the countries which don’t produce oil. \\
It is important to remark that this formula is a conservation law.

%---%
\section{ Model resolution: Optimal control approach}\label{modelres}
\noindent In this section, we give a necessary optimal condition to minimize the super coasts, denoted\footnote{For more details  in optimal control theory see \cite{SoZo1999}, \cite{Ru1994}, \cite{Ma2002}, \cite{Fau1988}}, $S(t).$ \\
A reasonable model to describe the evolution of the reserves of oil in the world can be written as follows: 
\begin{equation}\label{weq1}
\frac{d R}{dt}         =  -v(t)   + \alpha(t) R(t), t\geq 0
\end{equation}
where $\alpha (t) \in [0, 1[$ is  the rate  of  the increase of
oil in the world, $R(t)$ are the reserves and $v(t)$ is the demand
in the world at time $t$. \\
In the following, we are going to take $\alpha (t)=
\alpha $ a real constant. But the theorem below is satisfied even if
$\alpha $ depends on the time $t$.
\begin{thm}
Let  $a(t)$ be   the demand of the country which doesn't
produce oil, then
a necessary  optimal condition  to minimize the super coasts for a
country which doesn't product oil is:
$$
a^{*}(t) =   \displaystyle\frac{c_0 \exp(-\alpha t)}{2(p(t) - 29)^{2}}.
$$
\end{thm}
\begin{proof}
Let $T$ be  a fixed time.
\begin{eqnarray}
\min\int_0^{T} S(t)^{m} dt,\quad m\in \mathbb{N} \\
s.c \left\{
\begin{array}{lclcl}
\frac{d R}{dt}        & = & -v(t)   + \alpha R(t)  & on \,\, [0, T]& \\
R(0)   & = & Q
\end{array}
\right.
\end{eqnarray}
where, $Q$ is the initial stock of the reserves,
$$
S(t)  =   ( p(t)  - 29)a(t),
$$
$$
v  =     a + w,
$$
where $a$ is the demand of country which doesn't produce oil.\\
Let's  take the Hamiltonian defined by:
\begin{equation}\label{weq2}
H(R(t),a(t),\lambda(t),t) = S(t)^{m} +\lambda(t) ( -v(t) + \alpha  R(t))
\end{equation}
and $\lambda (t)$ being  the Pontryagin multiplier.\\
We have the  necessary optimal  conditions:
\begin{eqnarray*}
\frac{\partial H}{\partial R}&  =  & \alpha  \lambda(t)= -\frac{d \lambda }{dt }\\
\frac{\partial H}{\partial  a} & =   & 0\\
\end{eqnarray*}
The  condition  $\frac{\partial H}{\partial  a}  =    0$    is  equivalent  to  $m  S(t)^{m-1}( p(t) -29) - \lambda =0.$ \\
This implies that:
$$
S(t) =  \left(\frac{\lambda}{m(p(t) - 29)}\right)^{\frac{1}{m-1}}
$$
Since: $$-\frac{d \lambda }{dt }= \alpha  \lambda(t),$$
then $\lambda (t)= c_0 \exp(-\alpha t)$ where $c_0$ is a real constant.\\
Finally:
$$ S(t) =  \left(\frac{c_0 \exp(-\alpha t)}{m(p(t) - 29)}\right)^{\frac{1}{m-1}}$$
If  $m = 2,$ we  obtain:
\begin{eqnarray*}
S(t)& = & \displaystyle\frac{c_0 \exp(-\alpha t)}{2(p(t) - 29)}\,\,\mbox{and by the  following equality:}\\
S(t)&  = &  ( p(t)  - 29)a(t) \mbox{ we have:} \\
a{*}(t) &=  &  \displaystyle\frac{c_0 \exp(-\alpha t)}{2(p(t) - 29)^{2}}\qed
\end{eqnarray*}
\end{proof}

\subsection{Initial condition variation}
We assume that $R(t)$ is Lipschitzian, so there exists a flow defined by:
$$
\begin{cases}
X \rightarrow X\; \text{space of phases}\\
x_0 \rightarrow \phi_t(x_0) = x(t)
\end{cases}
$$
with $X=\mathbb{R}^{n}$, for 
n= 1,
$$
\begin{cases}
x(t) \rightarrow R(t) \\
x_0 \rightarrow Q
\end{cases}
$$
We consider a period of horizon h 
on the time interval $[t_0, t_0+h]$ and $Q \in [Q_0, Q^*]$
with $ Q_0 = Q(t0)$ and $ Q^* = Q(t_{t0+h})$.\\
We define:
$$
Q_k = Q_0 + \frac{k}{t_0+h}(Q^*-Q_0)
$$
The initial conditions vary with the variants of $k\in [t_0, t_0+h]$.

\subsection{Terminal condition variation}
We still consider equations (3.3) and (3.4) and rewrite them by varying the terminal conditions. Let $s \in [0, T]$. Let $t=T-s$, $Y(s) = R(T-s)$. The optimization problem becomes:
$$
min \int_{0}^{T}[S(T-s)]^mds
$$

$$
sc
\begin{cases}
\frac{dY(s)}{ds} = -v(T-s) + \alpha Y(s)\\
Y(0) = R_T
\end{cases}
$$
where $R_T$ is a real value sequence on $[Y_0,..., Y^*]$
with $Y_0 = Y(t_0)$ and $Y^* = Y(t_0 + h)$.

$$
R_T(k) = Y_0 + \frac{k}{t_0 +h}\left( Y^* - Y_0\right) 
$$
$$
S(T-s) = \left( P(T-s) - 29\right) a(T-s)
$$
\noindent In the next section we are going to use the curve below. This
represents the plot of  the result of the necessary optimal
condition. The  horizontal axis being the values  of the prices
while the vertical axis is the values of the demand $a^*$.

%---%
\section{Numerical resolution and results interpretations}\label{numericalres}
The simulation parameters are as follows. In the graphs of figure 1, we have taken as constant c the average of the percentage rate increase of the world fossil reserves. In figure 2, we consider the evolution of the rate of the resources over the study period in the numerical resolution of our optimal control problem
The curves in Figures 1-2 are the results of the optimal control resolution. 
The result could be interpreted as follows:
\begin{enumerate}
\item Between 1980 and 2021, [\ref{wadefig4}], there were two periods where the average price of oil was above 29 dollars, 1980 to 1983 and 2004 to 2021, and a period where it was below 29 dollars. In the first period (1980 to 1983), the average price of a barrel of oil fluctuated between  30 dollars and  37 dollars. It peaked at 39 dollars in 1981 due to the Reagan cut taxes.
Two other causes of the high fluctuation of the average barrel price are the Iran embargo in 1980 and the end of the recession in 1982.
\item During the second period (2004-2021), the average price of a barrel of oil fluctuated between about 35 dollars in 2004 and 103 dollars in 2011. It reached a peak of 127 dollars in 2008 during the financial crisis. Other factors that explain the high fluctuation during this period are Hurricane Katrina in 2005, Bernanke becomes Fred Chair in 2006, the banking crisis in 2007, the great recession in 2009, Iran threatening the Straits of Hormuz, the increase of the 15 \% dollars ,  the US shale oil increased 2015, the decrease of the dollars  in 2016, the OPEC cut oil supply to keep prices stable and demand reduction of the pandemic since 2020.
\item From 1984 to 2003, the average price of a barrel oscillated between 14 dollars  in 1986 and 27 dollars  in 2003. During the Gulf War, the peak was reached with a barrel at almost 33 dollars . Other events during this period allowed this fluctuation are the prices doubled, the recession and the war in Afghanistan in 2002.  The barrel reached the lowest price of 10 dollars in 1999.
\item The curves [\ref{wadefig3}] and  [\ref{wadefig4}] have the shape of the sigmoid function. The world's demand for fossil energy cannot grow in an unlimited way. With climate change and the strong fluctuations of the oil price, many countries are diversifying their energy sources toward renewable energies. Fossil fuel reserves are also limited despite the discovery of new deposits.
\item The evolution of the Super profit (curve c fig1) is consecutive to the differential between the price 29 and the average price oil with a phase of strong fall between 1984 and 2003
\item The curves [\ref{wadefig5}] and  [\ref{wadefig6}] have the shape of the sigmoid function. The world's demand for fossil energy cannot grow in an unlimited way. With climate change and the strong fluctuations of the oil price, many countries are diversifying their energy sources toward renewable energies. Fossil fuel reserves  [\ref{wadefig7}] are also limited despite the discovery of new deposits.
\item The countries which don't product oil have  to diversify their energy sources as soon as possible. It would be a  good
orientation even for the countries which product oil. Because the resources will disappear at a time $T^*$.
\item The strategies  to develop the research  in the other type of energy and their production have to be encouraged. This assertion
is justified in the following sense: when we plot the necessary optimal condition (see Figure1 a)\footnote{The same interpretation can be done for the real data [Figure 1b)]}, it is easy  to see that, if the price of the barrel goes far from 29 dollars, $a^*$ decreases. And then  the super coast $S(t)$ is minimized. But this quantity will not suffice for the  consumption  in the country because of growth of the needs in energy.
\end{enumerate}
%-%
\begin{figure}[h]
\subfloat[Simulated data $(c0= \%\bar\alpha =0.2$, m=2)\label{wadefig1}]{
\includegraphics[width=0.48\textwidth]{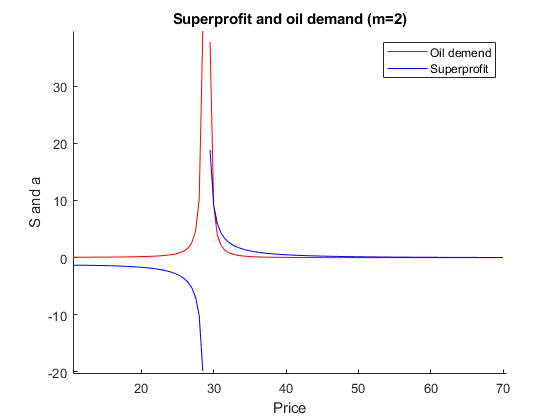}
}
\hfill
\subfloat[Evolution of world demand oil\label{wadefig2}]{
\includegraphics[width=.48\textwidth]{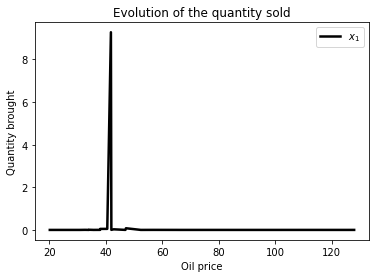}
}
\label{Wade Curve}
\vfill
\subfloat[S(t) Real data from  from International Energy Agency (IEA)\label{wadefig3}]{
\includegraphics[width=.48\textwidth]{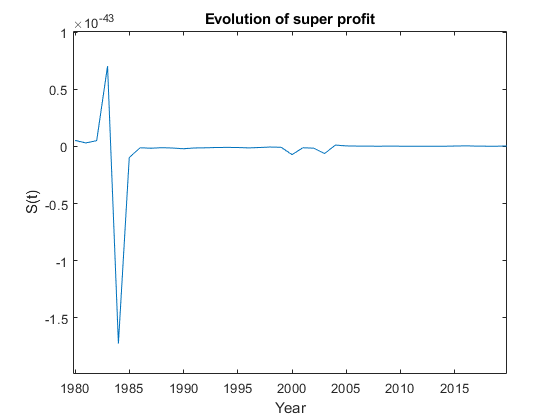}
}
\hfill
\subfloat[Oil price  data from  EIA\label{wadefig4}]{
\includegraphics[width=.48\textwidth]{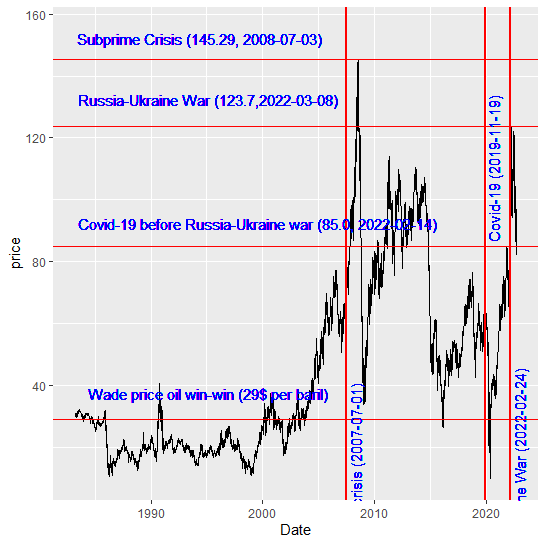}
}
\caption{Wade Curve}
\end{figure}
%-%
\begin{figure}[h]
\subfloat[oil wold demande optimal control result\label{wadefig5}]{
\includegraphics[width=0.48\textwidth]{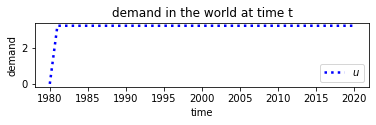}
}
\hfill
\subfloat[Objective optimal control function\label{wadefig6}]{
\includegraphics[width=.48\textwidth]{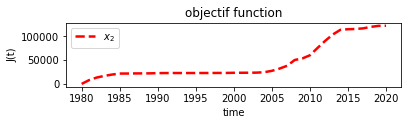}
}
\label{Wade Curve1}
\vfill
\subfloat[Reserves real data from IEA\label{wadefig7}]{
\includegraphics[width=.5\textwidth]{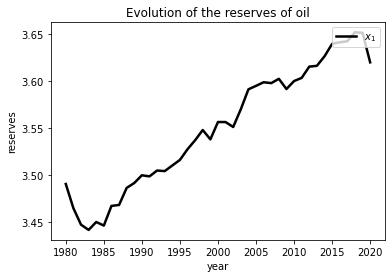}
}
\hfill
\label{Wade Curve2}
\ \\
\noindent \textit{Notes: The horizontal lines in Fig. 2 represent the maximum price of a barrel of oil during the different crises (Subprime around \$146, covid19 before the Ukrainian war, \$85 and Russo-Ukrainian war around \$124).The vertical lines are the beginning of the different cries}	
\end{figure}

%---%
\section{Concluding and Remarks}\label{ccl}
The Wade Formula  is also  satisfied if one day in the future the 29 dollars are not reasonable like a win win price. It suffices to replace 29  by $P_0(t)$ where $P_0(t)$  could satisfy the following system of equation and  without financial speculation 
$\frac{d P_0(t)}{dt }= f(i(t)) - \mu P_0(t)$ and $P_0(t)= i(t)+pr$. The index $i$ is the  investment per barrel, $ \mu \in [0,1[$ is the depreciation rate, $pr$ represents the  reasonable profit and $f$ is a function to be determined.

%---%

\end{document}